\documentclass[a4paper]{amsart}

\usepackage{cite}
\usepackage{amsmath}
\usepackage{amssymb}
\usepackage{amsthm}   
\usepackage{empheq}   
\usepackage{graphicx}
\usepackage{color}
\usepackage{appendix}

\theoremstyle{definition}
\newtheorem{theorem}{Theorem}[section]
\newtheorem{lemma}[theorem]{Lemma}
\newtheorem{definition}[theorem]{Definition}
\newtheorem{example}[theorem]{Example}

\newtheorem{proposition}[theorem]{Proposition}
\newtheorem{corollary}[theorem]{Corollary}
\theoremstyle{remark}

\numberwithin{equation}{section}

\def\EE{{\mathcal{E}}}
\def\FF{{\mathcal{F}}}

\def\SSS{{\mathbf{S}}}
\def\RRR{{\mathbf{R}}}
\def\GGG{\mathbf{G}}

\def\ss{\mathtt{s}}

\begin{document}

\title[Regular Dirichlet subspaces and Mosco Convergence]{Regular Dirichlet subspaces and Mosco Convergence}

\author{ Xiucui Song}
\address{School of Mathematical Sciences, Fudan University, Shanghai 200433, China.}
\email{xiucuisong12@fudan.edu.cn}

\author{Liping Li}
\address{School of Mathematical Sciences, Fudan University, Shanghai 200433, China.}
\email{lipingli10@fudan.edu.cn}





\begin{abstract}
 In this paper, we shall explore the Mosco convergence on regular subspaces of one-dimensional irreducible and strongly local Dirichlet forms. We find that if the characteristic sets of regular subspaces are convergent, then their associated regular subspaces are convergent in sense of Mosco. Finally, we shall show some examples to illustrate that the Mosco convergence does not preserve any global properties of the Dirichlet forms.
\end{abstract}

\maketitle

\section{Introduction}\label{SEC1}

What we are concerned in this paper is the theory of Dirichlet forms. The Dirichlet form was first raised by A. Beurling and J. Deny \cite{BD59} in 1959. Then M. Fukushima proved that the regular Dirichlet forms always possess associated symmetric Hunt processes in his excellent historical works (e.g. \cite{F71} \cite{F71-2}) at the beginning of 1970s. This sets up an exact connection between analysis and probability. On the other hand, M. Fukushima and J. Ying introduced a new conception, named by ``regular Dirichlet subpsace'', in 2003, see \cite{FY03} \cite{FY04}. Roughly speaking, for a given Dirichlet form, a regular Dirichlet subspace is its closed subspace with Dirichlet and regular properties. In 2005 and 2010, M. Fukushima, J. Ying and their co-authors characterized the regular Dirichlet subspaces of 1-dim Brownian motions and 1-dim irreducible diffusions by using a special class of scaling functions, see \cite{FFY05} and \cite{FHY10}. Furthermore, the second author of this paper, with J. Ying together, made more deep descriptions about the regular Dirichlet subspaces, such as \cite{LY14} \cite{LY15}. 

To introduce the conception of regular Dirichlet subspace, we need to explain the basic settings of Dirichlet forms briefly. Let $E$ be a measurable space and $m$ a $\sigma$-finite measure on $E$. Naturally, $L^2(E,m)$ is a real Hilbert space, whose norm and inner product are denoted by $\|\cdot\|_m$ and $(\cdot, \cdot)_m$. A Dirichlet form on $L^2(E,m)$ is usually written as $(\EE,\FF)$. Its definition is standard, see \cite{CF12} and \cite{FOT11}. In particular, if $E$ is a locally compact separable metric space and $m$ is a fully supported Radon measure on $E$, then we may talk about the regularity of Dirichlet forms. Indeed, a Dirichlet form $(\EE,\FF)$ on $L^2(E,m)$ is said to be regular, if $\FF\cap C_c(E)$ is dense in $\FF$ with the norm $\|\cdot\|_{\EE_1}$ and dense in $C_c(E)$ with the uniform norm, where the norm $\|\cdot\|_{\EE_1}$ corresponds to the inner product
\[
	\EE_1(u,v):=\EE(u,v)+(u,v)_m,
\]
and $C_c(E)$ is the class of continuous functions with compact supports on $E$. Moreover, the Borel measurable structure on $E$ is denoted by $\mathcal{B}(E)$. Without loss of generality, we use $f\in \mathcal{B}(E)$ to represent that the function $f$ is Borel measurable. Thus $\mathcal{B}(E)$ is formally the class of all Borel measurable functions on $E$. Furthermore, let $\mathcal{B}_+(E)$ and $b\mathcal{B}(E)$ be all positive Borel measurable functions and bounded Borel measurable functions on $E$ respectively. On the other hand, $C(E)$ is the class of all continuous functions on $E$, and $C_b(E), C_0(E)$ are its subspaces of bounded functions and being $0$ at infinity. In particular, $C^1(\RRR)$, $C^\infty(\RRR)$ are usual notations. 
We refer more terminologies of Dirichlet forms and potential theory to \cite{FOT11}.

Let $(\EE,\FF)$ and $(\EE',\FF')$ be two regular Dirichlet forms on $L^2(E,m)$. We say $(\EE',\FF')$ is a regular Dirichlet subspace, or a regular subspace in abbreviation, of $(\EE, \FF)$ provided that 
\[
	\FF'\subset \FF,\quad \EE(u,v)=\EE'(u,v),\quad u,v\in \FF'. 
\]
Furthermore, if $\FF'\neq \FF$, then $(\EE',\FF')$ is a proper regular subspace of $(\EE,\FF)$. In particular, we use
\[
	(\EE',\FF')\prec (\EE,\FF)
\]
to stand for that $(\EE',\FF')$ is a regular subspace of $(\EE,\FF)$. 

Our another focus in this paper is the Mosco convergence of Dirichlet forms. This kind of convergence was first introduced by U. Mosco \cite{M94} in 1994 and then widely used by lots of researchers, for instance \cite{BBCK09} \cite{K06} \cite{SU14}. In particular, it was also employed in \cite{BR14} to study the stochastic averaging principle of Halmiton dynamical system. 

Next, we shall briefly introduce the basic definition and some probabilistic significances of Mosco convergence. For any Dirichlet form $(\EE,\FF)$ on $L^2(E,m)$, we always extend the domain of $\EE$ to $L^2(E,m)$ by 
\[
	\EE(u,u)=\infty, \quad u\in L^2(E,m)\setminus \FF. 
\]
The following definition is given by U. Mosco \cite{M94}, in which that $u_n$ converges to $u$ weakly in $L^2(E,m)$ means that for any $v\in L^2(E,m)$, $(u_n,v)_m\rightarrow (u,v)_m$ as $n\rightarrow \infty$, and strong convergence means $\|u_n-u\|_m\rightarrow 0$ as $n\rightarrow \infty$. 

\begin{definition}\label{DEF141}
Let $\{(\EE^n,\FF^n):n\geq 1\}$ be a sequence of Dirichlet forms and $(\EE, \FF)$ another Dirichlet form  on $L^2(E,m)$. Then $(\EE^n, \FF^n)$ is said to be convergent to $(\EE, \FF)$ in sense of Mosco as $n\rightarrow \infty$, if
	\begin{description}
	 \item[(a)] for any sequence $\{u_n: n\geq 1\}$ of functions in $L^2(E,m)$, which is convergent to another function $u\in L^2(E,m)$ weakly, it holds that
	 \begin{equation}\label{141a}
	  	\liminf_{n\rightarrow \infty} \EE^n(u_n,u_n)\geq \EE(u,u);
	 \end{equation}
	 \item[(b)] for any function $u\in L^2(E,m)$, there always exists a sequence $\{u_n:n\geq 1\}$ of functions in $L^2(E,m)$, which is convergent to $u$ strongly as $n\rightarrow \infty$, such that
	 \begin{equation}\label{141b}
	  	\limsup_{n\rightarrow \infty} \EE^n(u_n,u_n)\leq \EE(u,u). 
	 \end{equation}
	\end{description}
\end{definition}

The most important significance of Mosco convergence is that it is equivalent to the convergence of associated semigroups. More precisely, let $\{T_t^n:t\geq 0\}$ and $\{G^n_\alpha:\alpha>0\}$ be the semigroup and resolvent of $(\EE^n,\FF^n)$, $\{T_t:t\geq 0\}$ and $\{G_\alpha: \alpha>0\}$ the semigroup and resolvent of $(\EE, \FF)$. Then $(\EE^n,\FF^n)$ is convergent to $(\EE,\FF)$ in sense of Mosco as $n\rightarrow \infty$, if and only if any one of following assertions holds: 
\begin{description}
 \item[(1)] for any $t>0,~ f\in L^2(E,m)$, $T_t^n f$ is convergent to $T_tf$ strongly in $L^2(E,m)$ as $n\rightarrow \infty$;
 \item[(2)] for any $\alpha>0,~ f\in L^2(E,m)$, $G_\alpha^n f$ is convergent to $G_\alpha f$ strongly in $L^2(E,m)$ as $n\rightarrow \infty$. 
\end{description}
Note that the semigoup of Dirichlet form is decided by the probability transition semigroup of associated Markov process. Hence the Mosco convergence implies the weak convergence of finite dimensional distributions of associated Markov processes. This fact is one of the reasons why the Mosco convergence is very useful in the theory of stochastic differential equations. 

At the end of this section, let us explain the structure of this paper. In \S\ref{SEC2}, we shall describe the associated Dirichlet forms of irreducible diffusions on 1-dimensional state space and characterize their regular subspaces. Particularly, we shall improve the results of \cite{FHY10} and give another description, say the characteristic sets, of regular subspaces. In \S\ref{SEC3} and \S\ref{SEC4}, we shall provide two conditions on characteristic sets to make the regular subspaces be Mosco convergent. Finally in \S\ref{SEC5}, we shall show some examples to claim that the Mosco convergence cannot maintain the stability of global properties of Dirichlet forms. The two convergence methods employed in \S\ref{SEC4} will be used.

\section{The regular subspaces of 1-dim irreducible diffusions and their charcteristic sets}\label{SEC2}

We always assume that $E$ is $\RRR$ or an open interval of $\RRR$, which is denoted by $I$. In other words, 
\[
	E=I=(a,b),
\]
where $-\infty \leq a<b\leq \infty$. 
The continuous stochastic process $X$ with strong Markov property on $I$ is also called a diffusion process. Further assume that
\begin{equation}\label{EQ31PXY}
	\mathbf{P}^x(\sigma_y<\infty)>0,\quad \forall x,y\in I,
\end{equation}
where $\sigma_y$ is the hitting time of $\{y\}$. 
This assumption, usually named by irreducibility, means that any two points of $I$ are connected for $X$ in intuition. Under this condition, the diffusion $X$ can be characterized completely by a strictly increasing and continuous function $\ss$, which is called the scaling function, and two Radon measures $m,k$ on $I$. In particular, $m$ is fully supported on $I$, and $X$ is $m$-symmetric. Furthermore, $k$ is the so-called killing measure of $X$, and we may assume that $k=0$ because of the studies in \cite{LY15}. Note that for any two constants $C, C_0$, if we replace $\ss, m$ by $C\cdot \ss+C_0, m/C$ respectively, then they still describe the same diffusion. We refer more details to \cite{IM74} and \cite{RW00}. 

\begin{definition}
The boundary point $a$ (resp. $b$) of $I$ is called $\ss$-approachable, if $\mathtt{s}(a+):=\lim_{x\downarrow a} \mathtt{s}(x)>-\infty$ \big(resp. $\mathtt{s}(b-):=\lim_{x\uparrow b} \mathtt{s}(x)<\infty$\big). Furthermore, $a$ (resp. $b$) is called an $\ss$-regular boundary, if $a$ is $\ss$-approachable and there is a constant $c\in I$ such that $m\big((a, c)\big)<\infty$ \big(resp. $b$ is $\ss$-approachable, and there is a constant $c\in I$ such that $m\big((c, b)\big)<\infty$\big).
\end{definition}

We always assume that the boundary $\{a,b\}$ of $I$ is the trap of $X$. That means if $X$ approaches the boundary, then it dies. Define
\begin{equation}\label{EQ2FSMU}
	\FF^{(\mathtt{s},m)}:=\bigg\{u\in L^2(I,m): u\ll \mathtt{s}, \frac{du}{d\mathtt{s}}\in L^2(I, d\mathtt{s})\bigg\},
\end{equation}
where $u\ll \mathtt{s}$ stands for that $u$ is absolutely continuous with respect to $\mathtt{s}$, or in other words, there is an absolutely continuous function $\varphi$ such that $u=\varphi\circ \mathtt{s}$. 
For any $u,v\in \FF^{(\mathtt{s},m)}$, set 
\[
	\EE^{(\mathtt{s},m)}(u,v):=\frac{1}{2}\int_I \frac{du}{d\mathtt{s}}\frac{dv}{d\mathtt{s}}d\mathtt{s}. 
\]
Note that $(\EE^{(\ss,m)},\FF^{(\ss,m)})$ is a Dirichlet form on $L^2(I,m)$ but not necessarily a regular one. One may prove that the associated Dirichlet form of diffusion $X$ on $L^2(I,m)$ can be written as
\begin{equation}\label{EQ31FSM}
	\begin{aligned}
		&\FF^{(\mathtt{s},m)}_0:=\{u\in \FF^{(\mathtt{s},m)}:  u(a)~\text{or}~u(b)=0, \text{if}~a~\text{or}~b~\text{is}~\ss\text{-regular}\}, \\
		&\EE^{(\mathtt{s},m)}(u,v)=\frac{1}{2}\int_I \frac{du}{d\mathtt{s}}\frac{dv}{d\mathtt{s}}d\mathtt{s}, \quad u,v\in \FF^{(\mathtt{s},m)}_0.
	\end{aligned}
\end{equation}
Moreover, $(\EE^{(\mathtt{s},m)}, \FF^{(\mathtt{s},m)}_0)$ is regular and irreducible with a special standard core
\[
	C_c^\infty\circ \mathtt{s}:=\{\varphi\circ \mathtt{s}: \varphi\in C_c^\infty(J)\},
\] 
where $J:=\mathtt{s}(I)=\{\mathtt{s}(x):x\in I\}$ is also an open interval of $\mathbf{R}$. Here, the irreducibility concerns Dirichlet forms (see \S1.6 of \cite{FOT11}), which differs to \eqref{EQ31PXY}. However, in this situation, the irreducibility of Dirichlet forms is equivalent to that of diffusion processes. The associated diffusion process of \eqref{EQ31FSM} is also called an absorbing diffusion or minimal diffusion.  

X. Fang, P. He and J. Ying in \cite{FHY10} first made a discussion about the regular subspaces of $(\EE^{(\mathtt{s},m)},\FF^{(\mathtt{s},m)}_0)$ and their global properties. But unfortunately, they did not assert that all regular subspaces of $(\EE^{(\mathtt{s},m)},\FF^{(\mathtt{s},m)}_0)$ can be described by the scaling functions, which were provided in \cite{FHY10}. Next, we shall give a brief proof to cover the above shortage. For that, take a fixed point $e$ on $I$, and set
\begin{equation}\label{EQ32SSI}
\begin{aligned}
	\mathbf{S}_\mathtt{s}(I):=\bigg\{\tilde{\mathtt{s}}: \tilde{\mathtt{s}}~&\text{is a strictly increasing and continuous function on}~I,\\& \tilde{\mathtt{s}}(e)=0,~\tilde{\mathtt{s}}\ll \mathtt{s},~\frac{d\tilde{\mathtt{s}}}{d\mathtt{s}}=0~\text{or}~1,~d\mathtt{s}\text{-a.e.}\bigg\}.
	\end{aligned}
\end{equation}
Note that the choice of $e$ is not essential. Since the scaling functions $\ss$ and $\ss+C$ describe the same diffusion process for any constant $C$, thus we fix the value of scaling function at a fixed point to avoid the presence of equivalence class.

\begin{proposition}\label{PRO1}
	For any $\tilde{\mathtt{s}}\in \mathbf{S}_\mathtt{s}(I)$, it holds that
\[
	(\EE^{(\tilde{\mathtt{s}},m)},\FF^{(\tilde{\mathtt{s}},m)}_0)\prec (\EE^{(\mathtt{s},m)},\FF^{(\mathtt{s},m)}_0).
\]
On the contrary, if $(\EE',\FF')\prec \big(\EE^{(\mathtt{s},m)},\FF^{(\mathtt{s},m)}_0\big)$, then there is a scaling function $\tilde{\mathtt{s}}\in \mathbf{S}_\mathtt{s}(I)$ such that
\[
	(\EE',\FF')=(\EE^{(\tilde{\mathtt{s}},m)},\FF^{(\tilde{\mathtt{s}},m)}_0).
\]
In particular, the regular subspace $(\EE^{(\tilde{\mathtt{s}},m)},\FF^{(\tilde{\mathtt{s}},m)}_0)$ is a proper one, if and only if $\tilde{\mathtt{s}}\neq \mathtt{s}$.  
\end{proposition}
\begin{proof}
The first and third assertions are the results of \cite{FHY10}, see Theorem~4.1 of \cite{FHY10}. 
	Now let $(\EE',\FF')$ be a regular subspace of $\big(\EE^{(\mathtt{s},m)},\FF^{(\mathtt{s},m)}_0\big)$. It follows from Theorem~4.1 of \cite{FHY10} that we only need to prove that $(\EE',\FF')$ is strongly local and irreducible. In fact, the strongly local property of $(\EE',\FF')$ is a corollary of Theorem~1 of \cite{LY15}. On the other hand, since $\big(\EE^{(\mathtt{s},m)},\FF^{(\mathtt{s},m)}_0\big)$ is irreducible and strongly local, from Theorem~4.6.4 of \cite{FOT11} and the definition of regular subspace, we can easily deduce that $(\EE',\FF')$ is also irreducible.
\end{proof}

In other words, the proposition above claims that the scaling function class $\SSS_\mathtt{s}(I)$ characterizes all regular subspaces of $\big(\EE^{(\mathtt{s},m)},\FF^{(\mathtt{s},m)}_0\big)$. Now we shall turn to introduce another equivalent description of $\SSS_\mathtt{s}(I)$. Set
\[
	\GGG_\mathtt{s}(I):=\bigg\{ G\subset I: \int_{G\cap (c,d)}d\mathtt{s}>0,~\forall c,d\in I, c<d\bigg\}. 
\]
Obviously, any set $G$ in $\GGG_\mathtt{s}(I)$ is defined in sense of $d\mathtt{s}$-a.e., in other words, it should be regarded as a $d\mathtt{s}$-a.e. equivalence class. The following lemma asserts that  $\GGG_\mathtt{s}(I)$ has an identical status with $\SSS_\mathtt{s}(I)$ for regular subspaces of $\big(\EE^{(\mathtt{s},m)},\FF^{(\mathtt{s},m)}_0\big)$. 

\begin{lemma}\label{LM1}
	There exists a bijective mapping between the scaling function class $\SSS_\mathtt{s}(I)$ and the class  $\GGG_\mathtt{s}(I)$ of sets. 
\end{lemma}
\begin{proof}
	For any $\tilde{\mathtt{s}}\in \SSS_\mathtt{s}(I)$, define
	\begin{equation}\label{EQ2GSX}
		G_{\tilde{\mathtt{s}}}:=\bigg\{x\in I: \frac{d\tilde{\ss}}{d\ss}(x)=1\bigg\}.
	\end{equation}	
Clearly $G_{\tilde{\ss}}$ is defined in sense of $d\ss$-a.e., and for any interval $(c,d)\subset I$, it holds that
\[
	\int_{G_{\tilde{\ss}}\cap (c,d)} d\ss=\int_c^d 1_{G_{\tilde{\ss}}}(x)d\ss(x)=\int_c^d d\tilde{\ss}>0.
\]
That implies $G_{\tilde{\ss}}\in \GGG_\ss(I)$. 

Now we shall prove that the mapping
\[
	\SSS_\ss(I)\rightarrow \GGG_\ss(I),\quad \tilde{\ss}\mapsto G_{\tilde{\ss}}
\]
is a bijective mapping. Firstly, it follows from $G_{\tilde{\ss}}\in \GGG_\ss(I)$ that this mapping is defined well. Secondly, let us prove that it is an injection. Assume that $\ss_1,\ss_2\in \SSS_\ss(I)$ satisfy $G_{\ss_1}=G_{\ss_2}$, $d\ss$-a.e. Then for any $x\in I$, 
\[
	\ss_1(x)=\int_e^x d\ss_1=\int_e^x \frac{d\ss_1}{d\ss}d\ss=\int_e^x 1_{G_{\ss_1}}(y)d\ss(y);
\]
Similarly, we have
\[
	\ss_2(x)=\int_e^x 1_{G_{\ss_2}}(y)d\ss(y).
\]
Hence we can deduce that $\ss_1=\ss_2$. Finally we shall explain that the mapping above is also a surjection. In fact, for any set $G\in \GGG_\ss$, let
\begin{equation}\label{EQ2SXE}
	\tilde{\ss}(x):=\int_e^x 1_G(y)d\ss(y),\quad x\in I. 
\end{equation}
We only need to prove $\tilde{\ss}\in \SSS_\ss(I)$ and $G_{\tilde{\ss}}\in \GGG_\ss(I)$. Indeed, from \eqref{EQ2GSX} we obtain that $\tilde{\ss}$ is strictly increasing. Furthermore, it follows from \eqref{EQ2SXE} that $\tilde{\ss}(e)=0$, $\tilde{\ss}\ll \ss$ and
\[
	\frac{d\tilde{\ss}}{d\ss} =1_G,\quad d\ss\text{-a.e.}
\]
This implies that $\tilde{\ss}\in \SSS_\ss(I)$ and $G_{\tilde{\ss}}\in \GGG_\ss(I)$. That ends the proof.  
\end{proof}

In Proposition~\ref{PRO1} and Lemma~\ref{LM1}, we obtain two equivalent characterizations of all regular subspaces of $\big(\EE^{(\ss,m)},\FF^{(\ss,m)}_0\big)$. For each regular subspace $\big(\EE^{(\tilde{\ss},m)},\FF^{(\tilde{\ss},m)}_0\big)$, the set $G_{\tilde{\ss}}$ in $\GGG_\ss(I)$, which corresponds to the scaling function $\tilde{\ss}$, is called the characteristic set of $\big(\EE^{(\tilde{\ss},m)},\FF^{(\tilde{\ss},m)}_0\big)$. Therefore, we may write down the following equivalent descriptions:
\begin{equation}\label{EQ2ESM}
	(\EE^{(\tilde{\ss},m)},\FF^{(\tilde{\ss},m)}_0)\leftrightharpoons (\EE^{(\tilde{\ss},m)},\FF^{(\tilde{\ss},m)})\leftrightharpoons \tilde{\ss}\leftrightharpoons G_{\tilde{\ss}}. 
\end{equation}
Note that the two Dirichlet forms in \eqref{EQ2ESM} are equal if and only if neither $a$ nor $b$ is $\ss$-regular.

The characteristic set is very important in the research of regular subspaces. For example, when $I=\RRR$, $m$ is the Lebesgue measure on $\RRR$ and $\ss$ is the natural scaling function, $\big(\EE^{(\ss,m)},\FF^{(\ss,m)}_0\big)$ is exactly the associated Dirichlet form of 1-dimensional Brownian motion. In another work of the second author and his co-author \cite{LY14}, they found that if the characteristic set $G$ is open, such as the complement of generalized Cantor set, then the regular subspace and Brownian motion share the same part on $G$. That means their difference concentrates on the boundary of $G$. This fact conduces to a study about the traces of Brownian motion and its regular subspace. We refer more details to \cite{LY14}. In the mean time, we denote a subset of $\GGG_\ss(I)$ by
 \begin{equation}
 	\overset{\circ}{\GGG}_{\ss}(I):=\{G\in \GGG_\ss(I): G~\text{has an open}~d\ss\text{-a.e.~version}\},
 \end{equation}
which will play an important role in what follows.

 \section{Mosco convergence \uppercase\expandafter{\romannumeral1}}\label{SEC3}

In this section, we shall consider the Mosco convergence on regular subspaces of $(\EE^{(\ss,m)},\FF^{(\ss, m)}_0)$. Before presenting the first convergence method, we need to prove a very useful lemma. Let $G\in \GGG_\ss(I)$ be a characteristic set and $F:=G^c$. Its associated scaling function is denoted by $\tilde{\ss}$. The following lemma provides another expression of Dirichlet form $(\EE^{(\tilde{\ss},m)}, \FF^{(\tilde{\ss},m)})$ from the viewpoint of characteristic set, a special case of which was already presented for Brownian motion in \cite{LY14}.
 
 \begin{lemma}\label{Lemma1}
 	It holds that
 	\begin{equation}\label{EQ3GUF}
 		\FF^{(\tilde{\ss},m)}=\bigg\{u\in \FF^{(\ss,m)}: \frac{du}{d\ss}=0,~d\ss\text{-a.e. on }F\bigg\}. 
 	\end{equation}
 \end{lemma}
 \begin{proof}
 	Note that $\FF^{(\tilde{\ss},m)}$ has the expression \eqref{EQ31FSM}. For any $u\in \FF^{(\tilde{\ss},m)}$, since $u\ll \tilde{\ss}$, it follows that $u\ll \ss$ and
 	\[
 		\frac{du}{d\ss}=\frac{du}{d\tilde{\ss}}\cdot \frac{d\tilde{\ss}}{d\ss}=\frac{du}{d\tilde{\ss}}\cdot 1_G,\quad d\ss\text{-a.e.}
 	\]
Thus we have  $du/d\ss=0$, $d\ss$-a.e. on $F$. On the contrary, assume that $u$ is a function in the class of right side of \eqref{EQ3GUF}, we only need to prove $u\ll \tilde{\ss}$ and $du/d\tilde{\ss}\in L^2(I, d\tilde{\ss})$. In fact, for any $x,y\in I$, 
 	\[
 		u(x)-u(y)=\int_y^x \frac{du}{d\ss}d\ss=\int_y^x \frac{du}{d\ss}\cdot 1_G d\ss=\int_y^x \frac{du}{d\ss} d\tilde{\ss}.
 	\]
 Hence $u\ll \tilde{\ss}$ and $du/d\tilde{\ss}=du/d\ss$, $d\tilde{\ss}$-a.e. It follows from $du/d\ss\in L^2(I,d\ss)$ that $du/d\ss\in L^2(I,d\tilde{\ss})$. That implies $u\in \FF^{(\tilde{\ss},m)}$, which completes the proof. 
 \end{proof}
 
 
Now, we assume that $\{G_n:n\geq 1\}$ is a sequence of sets in $\GGG_\ss(I)$. For each $n$, $G_n$ corresponds to the scaling function $\ss_n$.  Set $(\EE^n,\FF^n):=(\EE^{(\ss_n,m)},\FF^{(\ss_n,m)})$. Take another set $G\in \GGG_\ss(I)$, its associated scaling function is $\tilde{\ss}$, and set $(\EE,\FF):=(\EE^{(\tilde{\ss},m)},\FF^{(\tilde{\ss},m)})$. The following theorem asserts that if a sequence of characteristic sets is decreasing to another characteristic set, then the sequence of their associated Dirichlet forms is convergent in sense of Mosco. 

\begin{theorem}\label{THM141}
 If $G_n\downarrow G$, $d\ss$-a.e., then $(\EE^n,\FF^n)$ is convergent to $(\EE,\FF)$ in sense of Mosco. 
\end{theorem}
\begin{proof}
Firstly, we claim that
\begin{equation}\label{EQ3FFF}
	\mathcal{F}\subset \cdots \FF^{n+1}\subset \FF^n\subset \cdots \FF^1 \subset \FF^{(\ss,m)}. 
\end{equation}
Indeed, for $n$ and $n+1$, from $G_{n}\subset G_{n+1}$, we can deduce that $F_{n+1}:=G_{n+1}^c\subset G_n^c:=F^n$.  It follows from Lemma~\ref{Lemma1} that $\FF^{n+1}\subset \FF^n$. Similarly, from $G\subset G_n$, we have $\FF\subset \FF^n$. Clearly, $\FF^1\subset \FF^{(\ss,m)}$. 

Secondly, we shall prove (b) of Definition~\ref{DEF141}. If $u\notin \FF$, then $\mathcal{E}(u,u)=\infty$. Clearly, \eqref{141b} is right. For any $u\in \FF$, let $u_n:=u\in \FF\subset \FF^n$. Obviously, $u_n$ is strongly convergent to $u$. Note that $(\EE^n, \FF^n)$ and $(\mathcal{E},\mathcal{F})$ are both regular subspaces of $(\EE^{(\ss,m)},\FF^{(\ss,m)})$. We have
\[
	\mathcal{E}^n(u_n,u_n)=\EE^{(\ss,m)}(u_n,u_n)=\mathcal{E}(u_n,u_n). 
\]
Particularly, $\EE^n(u_n,u_n)=\mathcal{E}(u,u)$. Hence $\limsup_{n\rightarrow \infty}\EE^n(u_n,u_n)=\mathcal{E}(u,u)$, which implies that (b) is proved.

Finally, we turn to prove (a) of Definition \ref{DEF141}. Assume $u_n$ is weakly convergent to $u$ in $L^2(E,m)$. Without loss of generality, we may assume that $u_n$ belongs to $\FF^n$. Or, $\EE^n(u_n,u_n)=\infty$, which implies that $u_n$ is useless in the left side of \eqref{141a}. Fix an integer $N$, for any $n>N$, it follows from $u_n\in \FF^n\subset \FF^N$ that $\{u_n:n\geq N\}\subset \FF^N$.  In particular, $\EE^n(u_n,u_n)=\EE^N(u_n,u_n)$.  Note that a sequence of the same Dirichlet form is convergent to itself in sense of Mosco. That implies that 
\begin{equation}\label{EQR4ENU}
	\liminf_{n\rightarrow \infty}\EE^n(u_n,u_n)=\liminf_{n\geq N, n\rightarrow \infty}\EE^n(u_n,u_n)=\liminf_{n\geq N, n\rightarrow \infty}\EE^N(u_n,u_n)\geq \EE^N(u,u). 
\end{equation}
If for some integer $N$, $u\notin \FF^N$, then $\liminf_{n\rightarrow \infty}\EE^n(u_n,u_n)\geq \EE^N(u,u)=\infty$. Naturally,
\[
	\liminf_{n\rightarrow \infty}\EE^n(u_n,u_n)\geq \EE(u,u). 
\]
Now, assume that for any $N$, $u\in \FF^N$. From Lemma~\ref{Lemma1}, we know that $u\in \FF^{(\ss,m)}$, and $du/d\ss=0$, $d\ss$-a.e. on $F_N$.  It follows from $G_N\downarrow G$ that $\cup_{N\geq 1}F_N=F$, where $F:=G^c$. Hence we can obtain that  $du/d\ss=0$, $d\ss$-a.e. on $F$. By using Lemma \ref{Lemma1} again, we can deduce that $u\in \FF$. Particularly, since $(\mathcal{E},\FF)$ and $(\EE^N,\FF^N)$ in \eqref{EQR4ENU} are both regular subspaces of $(\EE^{(\ss,m)},\FF^{(\ss,m)})$, it follows that $\EE^N(u,u)=\mathcal{E}(u,u)$. From \eqref{EQR4ENU}, we obtain that
\[
	\liminf_{n\rightarrow \infty} \EE^n(u_n,u_n)\geq \EE(u,u),
\]
which completes the proof. 
\end{proof}

Although $(\EE^n, \FF^n)$ and $(\EE,\FF)$ in Theorem~ \ref{THM141} have the relation \eqref{EQ2ESM} with the corresponding regular subspaces, they are not exactly the regular subspaces. Next, we shall discuss some examples of Mosco convergence of real regular subspaces. 
Particularly, if there is a constant $c\in I$ such that $m\big((a,c)\big)=m\big((c,b)\big)=\infty$ (we use $m(a+)=m(b-)=\infty$ to stand for this property), then any irreducible diffusion on $I$ with speed measure $m$ would not have a regular boundary. Hence, the Dirichlet spaces in \eqref{EQ2FSMU} and \eqref{EQ2ESM} are the same. On the other hand, the speed measure is not essential for the structure of regular subspaces. In \cite{LY15}, we found that after a time change with full quasi support, the structure of regular subspaces maintains.

\begin{corollary}\label{COR31}
	We make the same assumptions as Theorem~\ref{THM141}, i.e. $G_n\downarrow G$, $d\ss$-a.e. If any one of following conditions is satisfied: 
\begin{description}
\item[(1)] $m(a+)=m(b-)=\infty$;
\item[(2)] if there is a constant $c\in I$ such that $d\ss\big(G\cap (a,c)\big)<\infty$~\big(resp. $d\ss\big(G\cap (c,b)\big)<\infty$\big), then there exists an integer $N$ such that $d\ss\big(G_N\cap (a,c)\big)<\infty$~\big(resp. $d\ss\big(G_N\cap (c,b)\big)<\infty$\big);
\end{description}
then $(\EE^{(\ss_n,m)},\FF^{(\ss_n,m)}_0)$ is convergent to $(\EE^{(\tilde{\ss},m)},\FF^{(\tilde{\ss},m)}_0)$ in sense of Mosco. 
\end{corollary}
 \begin{proof}
The sufficiency of first condition is clear. We only prove the sufficiency of second one. In fact, it suffices to prove
 \begin{equation}\label{EQ3FSM}
 	\FF^{(\tilde{\ss},m)}_0\subset \cdots \FF^{(\ss_n,m)}_0\subset \cdots \FF^{(\ss_1,m)}_0\subset \FF^{(\ss,m)}_0
 \end{equation}
 and
 \begin{equation}\label{EQ3FSM2}
 	\FF^{(\tilde{\ss},m)}_0=\bigg\{u\in \FF^{(\ss_N,m)}_0: \frac{du}{d\ss}=0,~d\ss\text{-a.e.}\bigg\}.
 \end{equation}
Indeed, from $G_{n+1}\subset G_n$, we have $\ss_{n+1}\ll \ss_n$, and $d\ss_{n+1}/d\ss_n=1$ or $0$, $d\ss_n$-a.e. Then it follows from Proposition~\ref{PRO1} that $\FF^{(\ss_{n+1},m)}_0\subset \FF^{(\ss_n,m)}_0$. Similarly, we can deduce that \eqref{EQ3FSM} is right. On the other hand, the second condition of Corollary~\ref{COR31} means that, $a$ or $b$ is $\tilde{\ss}$-regular, if and only if it is $\ss_N$-regular. That is because, if $a$ is $\ss_N$-regular, then from
 \[
 	\tilde{\ss}\ll \ss_N, \quad \frac{d\tilde{\ss}}{d\ss_N}=1~\text{or}~0,~d\ss_N\text{-a.e.},
 \]
we can deduce that $a$ is also $\tilde{\ss}$-regular; on the contrary, if $a$ is $\tilde{\ss}$-regular, which implies that $m(a+)<\infty$, and there is a constant $c\in I$ such that $d\ss\big(G\cap (a,c)\big)<\infty$, then from the second condition, we may obtain that $a$ is $\ss_N$-regular. Therefore we can complete the proof of \eqref{EQ3FSM2}, which is similar to that of Lemma~\ref{Lemma1}. 
 \end{proof}
  
  In particular, if $a$ and $b$ are both $\ss$-regular boundaries, then $a$ and $b$ are also $\tilde{\ss}$-regular and $\ss_n$-regular. Thus the second condition in Corollary~\ref{COR31} is naturally satisfied. At the end of this section, we shall give another example to show that if two conditions above are not satisfied, then $G_n\downarrow G$, $d\ss$-a.e. may not imply that $(\EE^{(\ss_n,m)},\FF^{(\ss_n,m)}_0)$ is convergent to $(\EE^{(\tilde{\ss},m)},\FF^{(\tilde{\ss},m)}_0)$ in sense of Mosco.

\begin{example}\label{EXA31}
Let $I=\RRR$, $\ss(x)=x$, and assume that $m(\RRR)<\infty$. Further assume that $G$ is the set, which is given by Example~5.2 of \cite{FHY10}. Clearly, $G\in \GGG_\ss(\RRR)$ and $|G|<\infty$, where $|\cdot|$ represents the Lebesgue measure on $\RRR$. Define
	\begin{equation}\label{EQ3GNG}
		G_n:=G \bigcup \bigg(\cup_{k\in \mathbf{Z}}(k-\frac{1}{n},k+\frac{1}{n})\bigg).
	\end{equation}
Let $\tilde{\ss}$ and $\ss_n$ denote the associated scaling functions of $G$ and $G_n$ respectively. Particularly, $\tilde{\ss}(-\infty)>-\infty, \tilde{\ss}(\infty)<\infty$. Note that $m(\RRR)<\infty$, which implies that  $-\infty$ and $\infty$ are both $\tilde{\ss}$-regular boundaries. Hence we can obtain 
\[
	\FF^{(\tilde{\ss},m)}\neq \FF^{(\tilde{\ss},m)}_0. 
\]
On the other hand, one may easily check that $G_n\in \GGG_\ss(\RRR)$ and 
	\[
		G_n\downarrow G, 	\quad \text{a.e.}
		\]
For each $n$, it follows that $\ss_n(-\infty)=-\infty, \ss_n(\infty)=\infty$. Thus $-\infty$ and $\infty$ are not $\ss_n$-regular boundaries, and
\[
	\FF^{(\ss_n,m)}=\FF^{(\ss_n,m)}_0. 
\]
Therefore, it follows from Theorem~\ref{THM141} that $(\EE^{(\ss_n,m)},\FF^{(\ss_n,m)}_0)$ is convergent to $(\EE^{(\tilde{\ss},m)},\FF^{(\tilde{\ss},m)})$ in sense of Mosco. However, $\FF^{(\tilde{\ss},m)}\neq \FF^{(\tilde{\ss},m)}_0$. Then by the uniqueness of Mosco convergence, we know that $(\EE^{(\ss_n,m)},\FF^{(\ss_n,m)}_0)$ cannot converge to $(\EE^{(\tilde{\ss},m)},\FF^{(\tilde{\ss},m)}_0)$ in sense of Mosco. 
\end{example}

\section{Mosco convergence \uppercase\expandafter{\romannumeral2}}\label{SEC4}

In \S\ref{SEC3}, we considered the Mosco convergence for decreasing characteristic sets. In this section, we shall discuss the increasing case. 

We first assert that Mosco convergence is invariant under spatial transforms of Dirichlet forms. More precisely, let
\[
	\{(\EE^n,\FF^n): n\geq 1\}
	\]
be a sequence of Dirichlet forms on $L^2(E,m)$, $(\EE,\FF)$ another Dirichlet form on $L^2(E,m)$.Moreover, $(\EE^n,\FF^n)$ is convergent to $(\EE,\FF)$ in sense of Mosco. Assume that $\hat{E}$ is another measurable space and 
\[
	j:E\rightarrow \hat{E},\quad x\mapsto \hat{x}
\]
is a measurable mapping. Let $\hat{m}:=m\circ j^{-1}$ be the image measure of $m$ with respect to $j$. Then
\[
	j^*: L^2(\hat{E},\hat{m})\rightarrow L^2(E,m),\quad  \hat{f}\mapsto \hat{f}\circ j
\]
is an isometric mapping, and the image space of $j^*$ is a closed subspace of $L^2(E,m)$. Set further 
\[
\begin{aligned}
	&\hat{\FF}:=\big\{\hat{f}\in L^2(\hat{E},\hat{m}): j^*\hat{f}\in \FF\big\},\\
	&\hat{\EE}(\hat{f},\hat{g}):=\EE(j^*\hat{f},j^*\hat{g}),\quad \hat{f},\hat{g}\in \hat{\FF}. 
\end{aligned}\]
If $j^*$ maps $L^2(\hat{E},\hat{m})$ onto $L^2(E,m)$, then $(\hat{\EE},\hat{\FF})$ is a Dirichlet form on $L^2(\hat{E},\hat{m})$. Similarly, we can define the image Dirichlet form $(\hat{\EE}^n,\hat{\FF}^n)$ of $(\EE^n,\FF^n)$ under  $j^*$.  

\begin{lemma}
Assume that $j^*$ is a surjection. On $L^2(\hat{E},\hat{m})$, the Dirichlet form $(\hat{\EE}^n,\hat{\FF}^n)$ is convergent to $(\hat{\EE},\hat{\FF})$ in sense of Mosco as $n\rightarrow \infty$. 
\end{lemma}
\begin{proof}
Denote the semigroups of $(\EE,\FF)$, $(\EE^n,\FF^n)$, $(\hat{\EE},\hat{\FF})$ and $(\hat{\EE}^n,\hat{\FF}^n)$ by $(T_t)_{t\geq 0}$, $(T^n_t)_{t\geq 0}$, $(\hat{T}_t)_{t\geq 0}$ and $(\hat{T}^n_t)_{t\geq 0}$ respectively. We only need to prove that for any $\hat{f}\in L^2(\hat{E},\hat{m})$ and $t\geq 0$, $\hat{T}^n_t\hat{f}$ converges to $\hat{T}_t\hat{f}$ strongly. In fact, since $j^*$ is surjective, one may easily check that
\[
	\hat{T}_t \hat{f}=T_t(j^*\hat{f})\circ j^{-1}, \quad \hat{T}^n_t \hat{f}=T^n_t(j^*\hat{f})\circ j^{-1}.
\]
Since $(\EE^n,\FF^n)$ is Mosco convergent to $(\EE,\FF)$, it follows that
\[
	||\hat{T}^n_t\hat{f}||_{\hat{m}}^2=\int_{\hat{E}}\bigg(T_t^n\big(j^*\hat{f}\big)\big(j^{-1}(\hat{x})\big)\bigg)^2\hat{m}(d\hat{x})=||T^n_t(j^*\hat{f})||_{m}^2\rightarrow ||\hat{T}_t\hat{f}||^2_{\hat{m}},
\]
which completes the proof. 
\end{proof}

If in addition, $j$ is a homeomorphism, then the Mosco convergences of Dirichlet forms and their transforms under $j$ are exactly equivalent. Note that the irreducible diffusion $X$ on $I$ with scaling function $\ss$  will be transformed to another irreducible diffusion with natural scaling function after spatial transform $\ss$. Thus without loss of generality, we shall always assume that $\ss$ is the natural scaling function on $I$ in this section. Let 
\[
	(\EE,\FF):=(\EE^{(\ss,m)},\FF^{(\ss,m)}_0). 
\]
Take a sequence of characteristic sets $\{G_n\in \overset{\circ}{\GGG}_\ss(I):n\geq 1\}$.  Furthermore, assume that all of them are open. For each $n$, denote the associated scaling function of $G_n$ by $\ss_n$, and set
\[
	(\EE^n,\FF^n):=(\EE^{(\ss_n,m)},\FF^{(\ss_n,m)}_0). 
\]
The following theorem is our main result of this section. 

\begin{theorem}\label{THM41}
If $G_n\uparrow I$, then the Dirichlet form $(\EE^n,\FF^n)$ is convergent to $(\EE,\FF)$ on $L^2(I,m)$ in sense of Mosco as $n\rightarrow \infty$. 
\end{theorem}
\begin{proof}
	Similarly to Theorem~\ref{THM141}, we can obtain that
	\[
		\FF^1\subset\cdots \subset \FF^n\subset \cdots\subset \FF. 
	\]	
Now we shall prove (a) of Definition~\ref{DEF141}. For any sequence $\{u_n:n\geq 1\}$ in $L^2(E,m)$, which is weakly convergent to $u$, we may always assume that $u_n\in\FF^n$. Or, $\EE^n(u_n,u_n)=\infty$. Then $u_n$ is useless in the left side of  \eqref{141a}. It follows that $u_n\in \FF^n\subset \FF$, and
\[
	\EE^n(u_n,u_n)=\EE(u_n,u_n). 
\]
Because of the same reason as that of the proof of Theorem~\ref{THM141}, we have 
\begin{equation}\label{EQR4NAU}
	\liminf_{n\rightarrow \infty} \mathcal{E}(u_n,u_n)\geq \mathcal{E}(u,u).
\end{equation}
Thus $\liminf_{n\rightarrow \infty} \mathcal{E}^n(u_n,u_n)\geq \mathcal{E}(u,u)$, i.e. (a) is proved. 

Finally, we turn to prove (b) of Definition~\ref{DEF141}. We assert that  $\cup_{n\geq 1}\FF^n$ is dense in $\FF$ with the norm $||\cdot ||_{\EE_1}$. Note that $C_c^\infty(G_n)\subset \FF^n$, and $C_c^\infty(I)$ is dense in $\FF$. For any function $u\in C_c^\infty(I)$, since the support of $u$ is compact, and
\[
	\text{supp}[u]\subset I=\cup_{n\geq 1}G_n,
\]
it follows that there is an integer $N$ such that $\text{supp}[u]\subset G_N$, which implies that $u\in C_c^\infty(G_N)\subset \FF^N$. Hence
\[
	C_c^\infty(I)\subset \cup_{n\geq 1} \FF^n. 
\]
Clearly, $\cup_{n\geq 1}\FF^n$ is dense in $\FF$. 

For any function $u\in L^2(E,m)$, if $u\notin \FF$, then \eqref{141b} is naturally satisfied. Now assume that $u\in \FF$. From the above assertion, we may find a sequence of functions $\{u_n:n\geq 1\}$ such that $u_n\in \FF^n$ and $||u_n-u||_{\EE_1}\rightarrow 0$ as $n\rightarrow \infty$. In particular, $u_n\in \FF$ and $\EE^n(u_n,u_n)=\EE(u_n,u_n)$. Therefore,
\[
	\limsup_{n\rightarrow \infty}\EE^n(u_n,u_n)=\lim_{n\rightarrow \infty} \EE(u_n,u_n)=\EE(u,u),
\]
which implies \eqref{141b}. That completes the proof. 
\end{proof}

\section{The instability of global properties under Mosco convergence}\label{SEC5}

The global properties of a Dirichlet form stand for its recurrence, transience, irreducibility, conservativeness and etc. We refer their standard definitions to \S1.6 of \cite{FOT11}.

K. Suzuki and T. Uemura in \cite{SU14} pointed out the following fact: Mosco convergence cannot maintain the stability of global properties of Dirichlet forms. In other words, a sequence of recurrent Dirichlet forms may converge to a transient one in sense of Mosco, and vice versa; a sequence of conservative Dirichlet forms may converge to a non-conservative one in sense of Mosco, and vice versa.  In this section, we shall present some examples in the context of regular subspaces to support their viewpoints in \cite{SU14}. That means we shall give the following examples:
\begin{description}
\item[(1)] a sequence of recurrent Dirichlet forms is convergent to a transient Dirichlet form in sense of Mosco;
\item[(2)] a sequence of transient Dirichlet forms is convergent to a recurrent Dirichlet form in sense of Mosco;
\item[(3)] a sequence of conservative Dirichlet forms converges to a non-conservative Dirichlet form in sense of Mosco;
\item[(4)] a sequence of non-conservative Dirichlet forms is convergent to a conservative Dirichlet form in sense of Mosco.
\end{description}
Before that, we need to point out some facts. The Dirichlet form $(\EE^{(\ss,m)},\FF^{(\ss,m)}_0)$, which is given by \eqref{EQ31FSM}, is transient if and only if $a$ or $b$ is $\ss$-approachable. The following example is about the instability of  recurrence/transience. 

\begin{example}\label{EXA51}
	Let $I=\RRR$, $m$ the Lebesgue measure on $\RRR$ and $\ss(x)=x$. In other words, $(\EE^{(\ss,m)},\FF^{(\ss,m)}_0)$ corresponds to 1-dimensional Brownian motion on $\RRR$. More precisely, 
	\[
		(\EE^{(\ss,m)},\FF^{(\ss,m)}_0)=\big(\frac{1}{2}\mathbf{D},H^1(\RRR)\big),
	\]
where $H^1(\RRR)$ is the 1-dim Sobolev space, and $\mathbf{D}(f,g):=\int_\RRR f'(x)g'(x)dx,~f,g\in H^1(\RRR)$. Clearly, this Dirichlet form is recurrent. 

Similarly to Example~\ref{EXA31}, let $G$ be the set given by Example~5.2 of \cite{FHY10},  and $G_n$ the characteristic set defined by \eqref{EQ3GNG}. 
Denote the associated scaling functions of $G$ and $G_n$ by $\tilde{\ss}$ and $\ss_n$. From Example~\ref{EXA31}, we know that
\[
	\tilde{\ss}(-\infty)>-\infty,\quad \tilde{\ss}(\infty)<\infty, 
\]
whereas for each $n$, 
\[
	\ss_n(-\infty)=-\infty,\quad \ss_n(\infty)=\infty. 
\]
That implies that $(\EE^{(\tilde{\ss},m)},\FF^{(\tilde{\ss},m)}_0)$ is transient, but $(\EE^{(\ss_n,m)},\FF^{(\ss_n,m)}_0)$ is recurrent. Note that $m$ satisfies the first condition of Corollary~\ref{COR31}. It follows that as $n\rightarrow \infty$, a sequence of recurrent Dirichlet forms $\{(\EE^{(\ss_n,m)},\FF^{(\ss_n,m)}_0):n\geq 1\}$ is convergent to a transient Dirichlet form $(\EE^{(\tilde{\ss},m)},\FF^{(\tilde{\ss},m)}_0)$ in sense of Mosco. 

Now we still take $G$ above. Note that $G$ is open. For any integer $n$, define
\begin{equation}\label{EQ5UNG}
	U_n:=G\cup (-n, n). 
\end{equation}
One may easily check that $U_n\in \overset{\circ}{\GGG}_\ss(\RRR)$, $\{U_n:n\geq 1\}$ is an increasing sequence of open sets, and $\cup_{n\geq 1}U_n=\RRR$. Denote the associated regular subspace of $U_n$ by $(\EE^n,\FF^n)$. It follows from Theorem~\ref{THM41} that $(\EE^n, \FF^n)$ is convergent to $\big(\frac{1}{2}\mathbf{D},H^1(\RRR)\big)$ in sense of Mosco. Finally, we assert that for each $n$, $(\EE^n,\FF^n)$ is transient. In fact, since $|G|<\infty$, it follows that $|U_n|<\infty$. Furthermore, its associated scaling function $\ss_n$ satisfies $\ss_n(-\infty)>-\infty,~\ss_n(\infty)<\infty$, which implies that $(\EE^n,\FF^n)$ is transient. 
\end{example}

We refer the definition of approachable boundary in finite time of 1-dimensional diffusion to Example~3.5.7 of \cite{CF12}. In particular, 
 $a$ (resp. $b$) is approachable in finite time, if and only if for some constant $c\in (a,b)$, 
\[
	\int_a^c m\big((x,c)\big)d\ss(x)<\infty,\quad (\text{resp.}~ \int_c^b m\big((c,x)\big)d\ss(x)<\infty).
\]
Apparently, regular boundary is always approachable in finite time. On the other hand, a minimal diffusion is conservative, if and only if neither $a$ nor $b$ is approachable in finite time. At the end of this paper, we shall present an example for the instability of conservativeness under Mosco convergence.  

\begin{example}
	We first set $I=\RRR$ and assume that $m(\RRR)<\infty$. The scaling function $\ss$ is the natural scaling function. Let $\{U_n:n\geq 1\}$ be the sequence \eqref{EQ5UNG} of sets in Example~\ref{EXA51}. Since $(\EE^{(\ss,m)},\FF^{(\ss,m)}_0)$ is recurrent, it is also conservative. We assert that the associated regular subspace of   $U_n$ is not conservative. In fact, since $m(\RRR)<\infty$ and the scaling function $\ss_n$, which corresponds to $U_n$, satisfies $\ss_n(-\infty)>-\infty,~\ss_n(\infty)<\infty$, it follows that $a$ and $b$ are both the approachable boundaries in finite time of $(\EE^n,\FF^n)$. In particular, $(\EE^n,\FF^n)$ is not conservative. Therefore, from Theorem~\ref{THM41}, we obtain that as $n\rightarrow \infty$, the non-conservative Dirichlet form $(\EE^n,\FF^n)$ is convergent to a conservative Dirichlet form $(\EE^{(\ss,m)},\FF^{(\ss,m)}_0)$ in sense of Mosco. 
	
Finally, we still set $I=\RRR$, $\ss$ is the natural scaling function and $m$ will be decided later. Assume that $G$ and $G_n$ are the characteristic sets in Example~\ref{EXA51}, $\tilde{\ss}$ and $\ss_n$ are their associated scaling functions. Set $(\tilde{a},\tilde{b}):=\tilde{\ss}(\RRR)$, then $\tilde{a}>-\infty, \tilde{b}<\infty$. Take a strictly increasing and integral function $F$ on $(\tilde{a},\tilde{b})$ such that
\[
	F(\tilde{a}+)=-\infty,\quad F(\tilde{b}-)=\infty.  
\]
The existence of $F$ is clear. For example, take a constant $0<\alpha<1$,  define $F(x):=1/|b-x|^\alpha$ near $\tilde{b}$ and $F(x):=-1/|x-a|^\alpha$ near $\tilde{a}$. Since $\tilde{\ss}$ is strictly increasing, it follows that $F\circ \tilde{\ss}$ is a strictly increasing function on $\RRR$. Without loss of generality, assume that $F\circ \tilde{\ss}(0)=0$. Furthermore, let $m$ be the Lebesgue-Stieltjes measure with respect to  $F\circ \tilde{\ss}$. In particular, we have
\[
	m\big((-\infty,0)\big)=m\big((0,\infty)\big)=\infty. 
\]
That implies that the first condition of Corollary~\ref{COR31} is satisfied. Thus the Dirichlet form $(\EE^{(\ss_n,m)},\FF^{(\ss_n,m)}_0)$ is convergent to $(\EE^{(\tilde{\ss},m)},\FF^{(\tilde{\ss},m)}_0)$ in sense of Mosco. Note that $(\EE^{(\ss_n,m)},\FF^{(\ss_n,m)}_0)$ is recurrent, hence they are all conservative (see Lemma~1.6.5 of \cite{FOT11}). In  the end, we assert that $(\EE^{(\tilde{\ss},m)},\FF^{(\tilde{\ss},m)}_0)$ is not conservative. It suffices to prove that $-\infty$ or $\infty$ is an approachable boundary of $(\EE^{(\tilde{\ss},m)},\FF^{(\tilde{\ss},m)}_0)$ in finite time. Indeed, since $F$ is integral, we can obtain that
\[
	\int_{-\infty}^0 m\big((x,0)\big)d\tilde{\ss}(x)=-\int_{-\infty}^0 F\circ \tilde{\ss}(x)d\tilde{\ss}(x)=\int_{\tilde{a}}^{\tilde{\ss}(0)}|F(y)|dy<\infty. 
\]
That implies that, as $n\rightarrow \infty$, the conservative Dirichlet form $(\EE^{(\ss_n,m)},\FF^{(\ss_n,m)}_0)$ is convergent to a non-conservative Dirichlet form $(\EE^{(\tilde{\ss},m)},\FF^{(\tilde{\ss},m)}_0)$ in sense of Mosco. 
\end{example}


\bibliography{mybib}




\end{document}